\documentclass[preprint,12pt]{elsarticle}

\usepackage{amssymb}
 \usepackage{amsthm}
 \usepackage{amsmath}
 \usepackage{graphicx}
 \usepackage{url}
 
 \usepackage{color}

\newtheorem{proposition}{Proposition}

\journal{ }

\begin{document}

\begin{frontmatter}



\title{Approximations of the integral of a class of sinusoidal composite functions}


\author[ecol]{Alberto Costa}
\ead{alberto.costa@sec.ethz.ch}

\address[ecol]{Future Resilient Systems, Singapore-ETH Centre, Singapore}

\begin{abstract}
Two approximations of the integral of a class of sinusoidal composite functions, for which an explicit form does not exist, are derived. Numerical experiments show that the proposed approximations yield an error that does not depend on the width of the integration interval. Using such approximations, definite integrals can be computed in almost real-time.
\end{abstract}

\begin{keyword}
approximation technique \sep Bessel function \sep numerical integration \sep sinusoidal composite functions


\end{keyword}

\end{frontmatter}

\section{Introduction}
It is well known that the explicit representation of some integrals in terms of known functions is not available. In such cases, numerical approximation methodologies as Newton-Cotes rules or the Gauss method can be employed to compute definite integrals \cite{handb,davra,numanal}. %(with Laguerre, Legendre, Hermite or Chebychev polynomials)

Among the functions for which the integral is not known, to the best of our knowledge there are the sinusoidal composite functions, e.g., $\cos(\cos(x))$, $\sin(\sin(x))$, which can appear in domains like electromagnetic waves and antennas analysis (see for instance \cite[Cap. 16, pp. 643-644]{applic}).
This has been confirmed by symbolic computation softwares such as Derive\textsuperscript{\texttrademark}6 \cite{derive}, Wolfram Mathematica\textsuperscript{\textregistered}8 \cite{wolfram}, the Symbolic Math Toolbox of MATLAB\textsuperscript{\textregistered}7.01 R14 \cite{matlab}, and Maple\textsuperscript{\texttrademark}15 \cite{maple}.\footnote{A note on the online integrator tool of Wolfram Mathematica\textsuperscript{\textregistered} at \url{https://reference.wolfram.com/language/tutorial/IntegralsThatCanAndCannotBeDone.html} states that for $\int{\sin(\sin(x))\,\mathrm{d} x}$ ``This integral cannot be done in terms of any of the standard mathematical functions built into the Wolfram Language.''. The computation of the integral of $\cos(\cos(x))$ produces a similar warning message.}

In this paper, two functions that approximate the integral of $\cos(\cos(x))$ are proposed. Using a similar methodology, approximations for other sinusoidal composite functions can be obtained. The advantages of having such approximations are: i) definite integrals can be computed almost real-time by evaluating the function at the extreme points of the integration domain, instead of employing more expensive numerical integration algorithms, and ii) such approximations can be used in contexts where the explicit formula is required, for example within non derivative-free optimization problems.

The rest of the paper is organized as follows: in Section \ref{sec:c}, the steps to derive the first approximation of the integral of $\cos(\cos(x))$ are introduced, including an error analysis. An improved approximation is then presented in Section \ref{sec:comcos}, while in Section \ref{sec:num} some numerical results are shown. Finally, Section \ref{sec:concl} summarizes conclusions and future work.

\section{The integration of $\cos(\cos(x))$}
\label{sec:c}
In this section, a functions that approximates the integral of $\cos(\cos(x))$ is derived. In order to do this, the following proposition about the integration of composite functions is introduced. As notation, the integral of $f(x)$ is $F(x)$, while its derivative is $f'(x)$.
 \begin{proposition}
 	Let $f(x)$ be an integrable function, and $g(x)$ a function for which we can compute both the derivative and the second order derivative. The following equations is valid:
 	\begin{equation}
 	\label{eq:formula}
 	\int{f(g(x))\,\mathrm{d} x}=\frac{F(g(x))}{g'(x)}-\displaystyle\int{F(g(x))\left(\frac{1}{g'(x)}\right)'\,\mathrm{d} x}+R,\;R\in\mathbb{R}.
 	\end{equation}
 \end{proposition}
 
 \begin{proof}
 	It is sufficient to derive both the left and right-hand sides.\end{proof}
 	%The equation holds by direct differentiation of both members.
 	%It is tedious to show how to obtain this formula; however, it is easy to check that it is correct: it is sufficient to derive both the left and right side to obtain the identity $f(g(x))=f(g(x))$. 
 An interesting similarity between Equation \eqref{eq:formula} and the integration \emph{by parts} of $\frac{f(x)}{g'(x)}$ can be noticed. As a matter of fact, using the integration by parts the following expression is obtained:
 \begin{equation}
 \label{eq:part}
 \int{\frac{f(x)}{g'(x)}\,\mathrm{d} x}=\frac{F(x)}{g'(x)}-\displaystyle\int{F(x)\left(\frac{1}{g'(x)}\right)'\,\mathrm{d} x}+R,\;R\in\mathbb{R},
 \end{equation}
 which has the same form of \eqref{eq:formula} if $F(x)$ is swapped with $F(g(x))$.  %Starting from equation \eqref{eq:formula}, and after some simplifications, we will obtain the approximation of the integral of $\cos(\cos(x))$.

The problem addressed in this paper can be more formally expressed as:
\begin{equation}
\label{eq:cos}
C(x)=\int{\cos\left(\cos\left(x\right)\right)\,\mathrm{d} x},
\end{equation}
where the aim is to find an approximation of $C(x)$. 
By applying \eqref{eq:formula} we obtain:
\begin{equation}
\label{eq:cosequat}
\int{\cos(\cos(x))\,\mathrm{d} x}=-\frac{\sin(\cos(x))}{\sin(x)}-\displaystyle\int{\sin(\cos(x))\left(\frac{-1}{\sin(x)}\right)'\,\mathrm{d} x}+R,\;R\in\mathbb{R}.
\end{equation}
It is not straightforward to compute the right-hand side integral, hence it should be approximated. As shown later, such approximation introduces a small and bounded error.
%\subsection{The first approximation}
The first step is to rewrite the integral as follows (for the sake of clarity, the integration constant $R$ is omitted from now on):
\begin{equation}
\int{\sin(\cos(x))\left(\frac{-1}{\sin(x)}\right)'\,\mathrm{d} x}=\int{\frac{\sin(\cos(x))\cos(x)}{\sin(x)^2}\,\mathrm{d} x}.
\end{equation}
Notice that by replacing $\sin(\cos(x))$ with $\cos(x)$, the problem can be reformulated as the computation of the integral of $\cot(x)^2$, whose solution is known.
It can be easily checked that $\sin(\cos(x))$ and $\cos(x)$ have the same maximum and minimum points, period and roots. The main difference, as shown by Figure \ref{fig:sincos}, is the value of the functions at the maximum and minimum points, but this can be taken into account by approximating $\sin(\cos(x))$ as follows:
\begin{equation}
\label{eq:approsinc}
\sin(\cos(x))\approx\frac{\sin(\cos(0))}{\cos(0)}\cos(x)=\sin(1)\cos(x).
\end{equation}  

\begin{figure}[h!]
	\centering
	\includegraphics[scale=0.8]{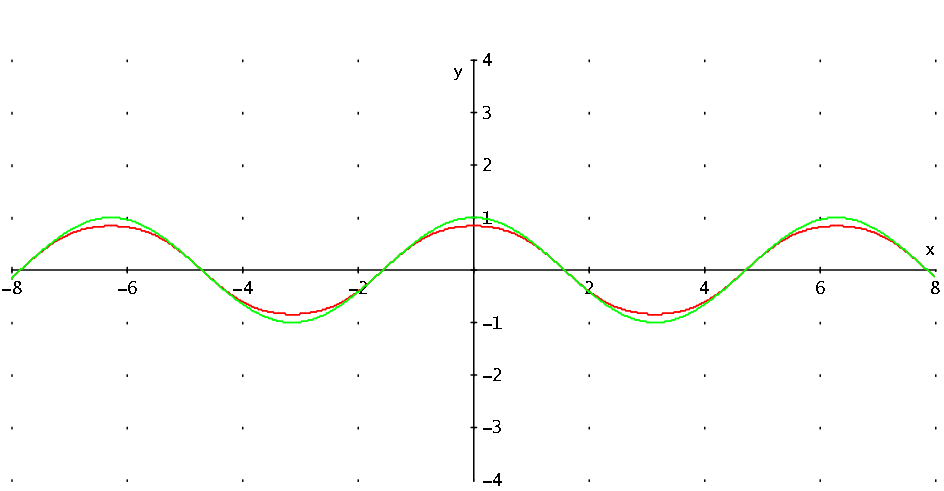}
	\caption{Plots of $\sin(\cos(x))$ (in red - dark) and $\cos(x)$ (in green - light).}
	\label{fig:sincos}
\end{figure}
Hence, we can write:
\begin{equation}
\int{\frac{\sin(\cos(x))\cos(x)}{\sin(x)^2}\,\mathrm{d} x}\approx\int{\sin(1)\cot(x)^2\,\mathrm{d} x}=-\sin(1)\left(\cot(x)+x\right).
\end{equation}
%where $\cot(x)$ is $\frac{1}{\tan(x)}$.
This way, Equation \eqref{eq:cosequat} becomes:
\begin{equation}
\label{eq:firstapp}
\int{\cos(\cos(x))\,\mathrm{d} x}\approx-\frac{\sin(\cos(x))}{\sin(x)}+\sin(1)\left(\cot(x)+x\right)=\tilde{C}(x).
\end{equation}

By comparing the plot of $\tilde{C}(x)$ and the plot of $C(x)$ (obtained with Derive\textsuperscript{\texttrademark}) in Figure \ref{fig:grap}, it is clear that the approximation $\tilde{C}(x)$ is not very accurate.
\begin{figure}[h]
	\centering
	\includegraphics[scale=0.8]{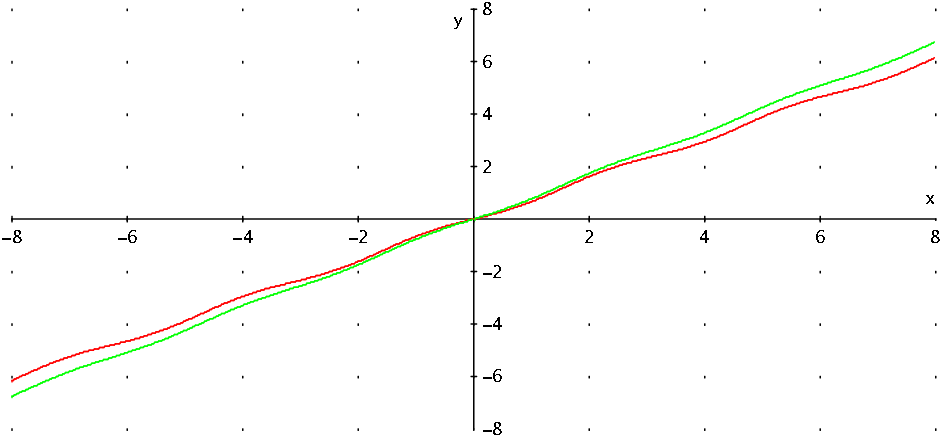}
	\caption{Plot of the integral of $\cos(\cos(x))$ (in red - dark) and $\tilde{C}(x)$ (in green - light).}
	\label{fig:grap}
\end{figure}
Looking at Equation \eqref{eq:firstapp}, it appears that $\tilde{C}(x)$ can be divided into two parts: the first part is a \emph{periodic} function $\tilde{P}(x)$, while the second one is a \emph{linear} function $\tilde{L}(x)$. More precisely, we have:
\begin{eqnarray}
\tilde{P}(x)&=&-\frac{\sin(\cos(x))}{\sin(x)}+\sin(1)\cot(x),\\
\tilde{L}(x)&=&\sin(1)x.
\end{eqnarray}
This subdivision is useful because $\tilde{C}(x)$ introduces some errors in both $\tilde{P}(x)$ and $\tilde{L}(x)$. In order to derive a better approximation, $\tilde{L}(x)$ is analyzed
in Section \ref{sec:l}, while in Section \ref{sec:p} $\tilde{P}(x)$ is considered.

\subsection{Linear function correction}
\label{sec:l}
From the analysis of Figure \ref{fig:grap} it appears that the slope of $\tilde{C}(x)$ is not the same of $C(x)$. In other words, the coefficient $\sin(1)$ in $\tilde{L}(x)$ is wrong.
In order to obtain a better coefficient for the term $x$ in $\tilde{L}(x)$, the following relationship is employed \cite{handb}:
\begin{equation}
\int_0^{\pi}{\cos(\cos(x))\,\mathrm{d}x}=\pi J_0(1),
\end{equation}
where $J_0(x)$ is the Bessel function of the first kind with order 0.
Let $\alpha$ be the correct coefficient of $x$. The value of $\alpha$ can be obtained by solving this equation:
\begin{equation}
\left[\tilde{P}(x)+\alpha x\right]_0^{\pi}=\pi J_0(1).
\end{equation}
%Since computing $\tilde{P}(x)$ in 0 and $\pi$  is not possible, we must employ the limit operator, obtaining:
Thus, we have:
\begin{equation}
\lim_{x\rightarrow \pi}\tilde{P}(x)-\lim_{x\rightarrow 0}\tilde{P}(x) + \alpha\pi=\alpha\pi={\pi J_0(1)},
\end{equation}
hence $\alpha=J_0(1)$. 
The linear function $\tilde{L}(x)$ can be replaced by $L(x)=J_0(1)x$, and our approximation becomes:
\begin{equation}
\hat{C}(x)=\tilde{P}(x)+L(x)=-\frac{\sin(\cos(x))}{\sin(x)}+\sin(1)\cot(x)+J_0(1)x.
\end{equation}
Numerically, it can be shown that $J_0(1)$ is the real slope of the linear component of $\int{\cos(\cos(x))\,\mathrm{d} x}$. In other words, the function $\int{\cos(\cos(x))\,\mathrm{d} x}-J_0(1)x$ is a periodic function, and the period is the same of $\tilde{P}(x)$.
\begin{proposition}
	\label{prop:slope}
	The function $\int{\cos(\cos(x))\,\mathrm{d} x}-J_0(1)x$ is periodic, with the same period of $\tilde{P}(x)$, that is $\pi$.
\end{proposition}
\begin{proof}
	A graphical hint that the linear component of $\int{\cos(\cos(x))\,\mathrm{d} x}$ is $J_0(1)x$ can be obtained by looking at Figure \ref{fig:bessel}. 
	\begin{figure}[h]
		\centering
		\includegraphics[scale=0.8]{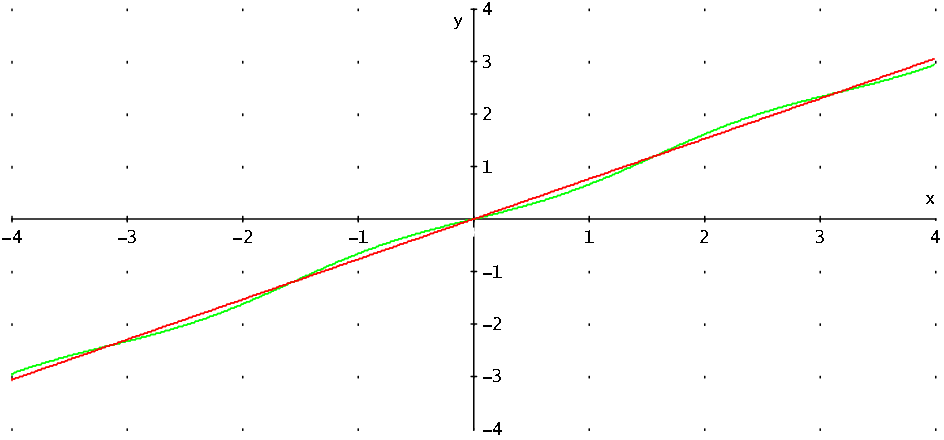}
		\caption{Plot of $J_0(1)x$ (in red - dark) and the integral of $\cos(\cos(x))$ (in green - light).}
		\label{fig:bessel}
	\end{figure}
	However, the following property can be checked numerically:
	\begin{equation}
	\int_{t}^{t+\pi}{\cos(\cos(x))\,\mathrm{d} x}-J_0(1)\pi=0,\,\,\forall t \in \mathbb{R},
	\end{equation}
	meaning that this is a periodic function and its period is $\pi$. Finally, it is easy to prove that the period of $\tilde{P}(x)$ is $\pi$. %the following: 
	%\begin{equation}
%	\tilde{P}(t)=\tilde{P}(t+\pi),\,\forall t \in\mathbb{R},
	%\end{equation}
	%so the period of $\tilde{P}(x)$ is $\pi$.
\end{proof} 

%The interesting fact of approximation $\hat{c}(x)$ is that we can give an upper bound on the maximum error which does not depend on the interval of integration.

\subsection{Periodic function correction}
\label{sec:p}
Approximation $\hat{C}(x)$ presents an error similar to the one shown in Figure \ref{fig:sincos}, where the amplitude of $\hat{C}(x)$ and $C(x)$ is not always the same. %, and the maximum error corresponds to $x\in\mathcal{C}=\left\{\frac{\pi}{4}+t\frac{\pi}{2},\,t\in\mathbb{Z}\right\}$.
In order to mitigate this error, a first idea is to multiply $\hat{C}(x)$ by a constant $k$. In Section \ref{sec:comcos} a more accurate method is proposed.
There are two important points to underline before finding this constant $k$. First, $L(x)$ should not be changed, since it represents the correct linear part of the integral, as explained by Proposition \ref{prop:slope}. This means that $k$ will multiply only $\tilde{P}(x)$, giving $P(x)=k\tilde{P}(x)$ which will appear in the final approximation.
The second remark is that if we want to proceed as done in \eqref{eq:approsinc}, $\hat{C}(x)$ and $C(x)$ should be obtained at the points where their difference is maximum, and this is challenging. %in $\frac{\pi}{4}$ are needed. While the former is not difficult to obtain, the latter introduces a challenge. As a matter of fact, the value of the integral of $\cos(\cos(x))$ in $x=\frac{\pi}{4}$ can be known by computing its definite integral from 0 to $\frac{\pi}{4}$ (the value of the integral in 0 is 0), but the value in $x=\frac{\pi}{4}$ needs to be obtained numerically. 
However, an alternative way to obtain the constant $k$ is to employ derivatives:
%\begin{equation}
%\label{eq:primadefk}
%\int{\cos(\cos(x))\,\mathrm{d} x}\approx k\tilde{P}(x)+L(x).
%\end{equation}
%After deriving both left and right hand sides of the equation, we obtain:
\begin{equation}
\frac{\mathrm{d}}{\mathrm{d} x}\int{\cos(\cos(x))\,\mathrm{d} x}\approx\frac{\mathrm{d}}{\mathrm{d} x}\left(k\tilde{P}(x)+L(x)\right),
\end{equation}
hence
\begin{equation}
\label{eq:k}
\cos(\cos(x))\approx k\left(\cos(\cos(x))+\frac{\cos(x)\sin(\cos(x))-\sin(1)}{\sin(x)^2}\right)+J_0(1).
\end{equation}
As a consequence, $k$ can be estimated by evaluating Equation \eqref{eq:k} at the points where the absolute value of the difference between $\cos(\cos(x))$ and $\tilde{P}'(x)+L'(x)$ is maximum. Such \emph{maximum difference} points belong to the set $D=\left\{t\frac{\pi}{2},\,t\in\mathbb{Z}\right\}$, and this can also be checked visually by Figure \ref{fig:derivc}.

\begin{figure}[h]
	\centering
	\includegraphics[scale=0.8]{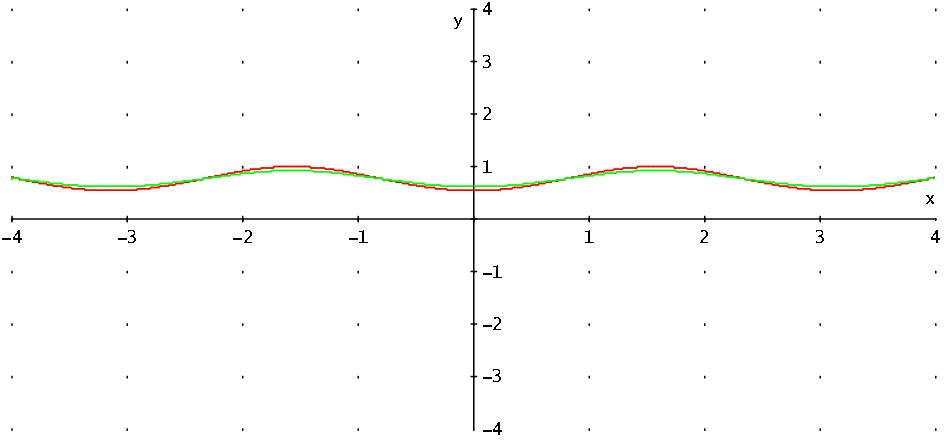}
	\caption{Plot of $\cos(\cos(x))$ (in red - dark) and $\tilde{P}'(x)+L'(x)$ (in green - light).}
	\label{fig:derivc}
\end{figure}
It should be pointed out that the absolute value of the difference between the functions in Figure \ref{fig:derivc} for $x\in D_a=\left\{t\pi,\,t\in\mathbb{Z}\right\}$ is not exactly the same as the difference for the points $x\in D_b=\left\{t\pi+\frac{\pi}{2},\,t\in\mathbb{Z}\right\}$. Hence, two constants $k_a$ and $k_b$ can be derived.

Consider now the estimation of $k_a$. As $0\in D_a$, $k_a$ can be obtained by solving Equation \eqref{eq:k} for $x=0$: %(since there is a 0 at the denominator, we employ the limit operator):
\begin{equation}
\cos(\cos(0))=k_a\cos(\cos(0))+k_a\lim_{x\rightarrow0}{\frac{\cos(x)\sin(\cos(x))-\sin(1)}{\sin(x)^2}}+J_0(1).
\end{equation}
The resolution gives:
%\begin{equation}
%\cos(1)=k_a\frac{\cos(1)-\sin(1)}{2}+J_0(1).
%\end{equation}
%Finally, the value of $k_a$ is:
\begin{equation}
\label{eq:kaa}
k_a=\frac{2\left(\cos(1)-J_0(1)\right)}{\cos(1)-\sin(1)}.
\end{equation}

Concerning the estimation of $k_b$, Equation \eqref{eq:k} can be solved for $x=\frac{\pi}{2}\in D_b$:
\begin{equation}
\cos\left(\cos\left(\frac{\pi}{2}\right)\right)=k_b\cos\left(\cos\left(\frac{\pi}{2}\right)\right)+k_b{\frac{\cos\left(\frac{\pi}{2}\right)\sin\left(\cos\left(\frac{\pi}{2}\right)\right)-\sin(1)}{\sin\left(\frac{\pi}{2}\right)^2}}+J_0(1).
\end{equation}
This provides:
\begin{equation}
\label{eq:kbb}
k_b=\frac{1-J_0(1)}{1-\sin(1)}.
\end{equation}

In order to reduce the average error of the approximation, a possible value for the constant $k$ is:
\begin{equation}
k=\frac{k_a+k_b}{2}=\frac{1}{2}\left(\frac{2\left(\cos(1)-J_0(1)\right)}{\cos(1)-\sin(1)}+\frac{1-J_0(1)}{1-\sin(1)}\right).%=\frac{\sin(2)+2J_0(1)-\left(\cos(1)+\sin(1)\right)(J_0(1)+1)}{2(\sin(1)-1)(\cos(1)-\sin(1))}.
\end{equation}
Finally, the first complete approximation becomes:
\begin{equation}
\label{finalcos}
C_1(x)=P(x)+L(x)=k\left(-\frac{\sin(\cos(x))}{\sin(x)}+\sin(1)\cot(x)\right)+J_0(1)x.
\end{equation}

\subsection{Error analysis}
\label{sec:erran}
When employing Formula \eqref{finalcos} for computing a definite integral, an error $\epsilon$ is introduced: %. Consider the computation of the following definite integral:
%\begin{equation}
%\int_l^u{\cos(\cos(x))\,\mathrm{d} x}=C(u)-C(l).
%\end{equation}
$$\epsilon=\left|\int_l^u{\cos(\cos(x))\,\mathrm{d} x}-C_1(u)+C_1(l)\right|=|C(u)-C(l)-C_1(u)+C_1(l)|.$$

By comparing $\cos(\cos(x))$ and $\tilde{P}'(x)+L'(x)$ it can be noticed that the maximum error occurs when $l=l_h=-\frac{\pi}{4} + h\frac{\pi}{2}$, $u=u_h=\frac{\pi}{4} + h\frac{\pi}{2}$, $\forall h\in\mathbb{Z}$ (see Figure \ref{fig:derivc}).
%both $l$ and $u$ are in the set of the maximum error points $\mathcal{C}$, and the errors have opposite signs for $l$ and $u$ (i.e., our approximation gives a value greater than the real one for $l$, and smaller for $u$, or vice-versa). This happens when $u-l=\frac{\pi}{4}+h\pi,\,h\in\mathbb{Z}$.
Since $L(x)$ represents the correct linear component of the integral, the error is in $P(x)=k\tilde{P}(x)$. %the maximum error points the real value of $k$ is between \eqref{eq:kaa} and \eqref{eq:kbb}.
%As $k=\frac{k_a+k_b}{2}$, and the real value of $k$ is between \eqref{eq:kaa} and \eqref{eq:kbb}, the maximum absolute error for the coefficient of $P(x)$ in the evaluation of $l$ is $\frac{k_a-k_b}{2}$, and the same is for $u$. 
%An upper bound for the error $\epsilon$ can be computed, where $\epsilon$ is defined as:
Let the real value (unknown in practice) of the integral of $\cos(\cos(x))$ computed at $x=t$ be $k(t) \tilde{P}(t)+L(t)$. Since $L(x)$ represents the correct linear component of the integral, as shown by Proposition \ref{prop:slope}, we can write:

\begin{equation}
\label{eq:err_def}
\epsilon=|(k(u)-k)\tilde{P}(u)-(k(l)-k)\tilde{P}(l)|.
%\epsilon=|C(b)-C(a)-C_1(b)+C_1(a)|=|(k_b^\ast-k)\tilde{P}(b)-(k_a^\ast-k)\tilde{P}(a)+L(b)-L(a)-L(b)+L(a)|.
\end{equation}

%Since $\tilde{P}(x)$ is a periodic function of period $\pi$, and $u-l=\frac{\pi}{4}+h\pi,\,h\in\mathbb{Z}$, then $\tilde{P}(u)=\tilde{P}(l+\frac{\pi}{4})$, and it is easy to check that $\tilde{P}(l+\frac{\pi}{4})=-\tilde{P}(l)$, thus we can conclude that $\tilde{P}(u)=-\tilde{P}(l)$. Moreover, for all points in $\mathcal{C}$ the function $\tilde{P}(x)$ has the same value (without considering the sign), so one of them can be taken, for example $\frac{\pi}{4}$, obtaining:
It can be easily checked that $\tilde{P}(x)=-\tilde{P}(-x)$, hence $\tilde{P}(u_h)=-\tilde{P}(l_h)$. By taking, for example, $h=0$, we have $l_0=-\frac{\pi}{4}$ and $u_0=\frac{\pi}{4}$, and using the triangle inequality we get:
\begin{equation}
\epsilon\leq\left|\tilde{P}\left(\frac{\pi}{4}\right)(k(u_0)-k+k(l_0)-k)\right|\leq\left|\tilde{P}\left(\frac{\pi}{4}\right)\right|\left(\left|k(u_0)-k\right|+\left|k(l_0)-k)\right|\right).
\end{equation}
%\leq\left|\tilde{P}(u_h)(k(u_h)-k+k(l_h)-k)\right|=

At this point, in order to to provide an upper bound for the error, recall that the value of $k(t)$ is between $k_a$ and $k_b$, and $k$ is $\frac{k_a+k_b}{2}$. Since $k_a>k_b$, we have:
\begin{eqnarray}
|k(l_0)-k|\leq\frac{k_a-k_b}{2},\\
|k(u_0)-k|\leq\frac{k_a-k_b}{2}.
\end{eqnarray}
Therefore, the following upper bound is obtained:
\begin{equation}\epsilon\leq\left|\tilde{P}\left(\frac{\pi}{4}\right)\right|(k_a-k_b).
\end{equation}
Numerically, the result is:
\begin{equation}
\epsilon\leq\left(\sqrt{2}\sin\left(\frac{\sqrt{2}}{2}\right)-\sin(1)\right)\left(\frac{2\left(\cos(1)-J_0(1)\right)}{\cos(1)-\sin(1)}-\frac{1-J_0(1)}{1-\sin(1)}\right)\leq 9.6\cdot 10^{-4},
\end{equation}
%\begin{equation}
%\epsilon\leq|\tilde{P}\left(\frac{\pi}{4}\right)(k_a-k_b)|=\left|\left(\sin(1)-\sqrt{2}\sin\left(\frac{\sqrt{2}}{2}\right)\right)\left(\frac{2\left(\cos(1)-J_0(1)\right)}{\cos(1)-\sin(1)}-\frac{1-J_0(1)}{1-\sin(1)}\right)\right|.
%\end{equation}
%After calculating this upper bound, we can finally write
%\begin{equation}
%\epsilon\leq0.000955,
%\end{equation}
which gives an idea about the precision of the approximation: in the worst case, the error is less than $10^{-3}$.

Indeed, this is a worst-case. The best-case scenario is when $l=h\frac{\pi}{2}$ and $u=\frac{\pi}{2}+h\frac{\pi}{2}$, $\forall h\in \mathbb{Z}$. In this case, it is easy to check that $\tilde{P}(l)=\tilde{P}(u)=0$, hence from \eqref{eq:err_def} the error is 0. %(without considering the approximation errors for the evaluation of the Bessel function).
%and $L(x)$ has the same value of the real integral, as can be seen by Figure \ref{fig:bessel}. This would yield an error of 0 (without considering the approximation errors for the evaluation of the Bessel function).

\section{Improved approximation}
\label{sec:comcos}
In Section \ref{sec:p} a constant $k$ is employed to obtain a better approximation of $\tilde{P}(x)$. As suggested in the error analysis of Section \ref{sec:erran}, a more accurate result could be achieved by replacing the constant $k$ with a function $k(x)$, but in this case finding the function $k(x)$ may prove to be difficult, since it would be the solution of the following differential equation:
\begin{equation}
\cos(\cos(x))=k'(x)\tilde{P}(x)+k(x)\tilde{P}'(x)+J_0(1).
\end{equation}

An easier way to estimate $k(x)$ comes from the considerations done in Section \ref{sec:p}. The function $\tilde{P}(x)$ can be multiplied by $k(x)$ such that:
\begin{align}
&k(t\pi)=k_a,\,t\in\mathbb{Z}\\
&k\left(t\pi+\frac{\pi}{2}\right)=k_b,\,t\in\mathbb{Z}.
\end{align}
A function that respects such conditions is:
%$k_a$ $\left\{x\,|\,x=t\pi,\,t\in\mathbb{Z}\right\}$, and $k_b$ in $\left\{x\,|\,x=t\pi+\frac{\pi}{2},\,t\in\mathbb{Z}\right\}$. It is easy to see that such a function can be the following:
\begin{equation}
k(x)=\frac{k_a-k_b}{2}\cos(2x)+\frac{k_a+k_b}{2}.
\end{equation}
After swapping $k$ with $k(x)$ in Equation \eqref{finalcos}, a new approximation is derived:
\begin{equation}
\label{finalcos2}
C_2(x)=\left(\frac{k_a-k_b}{2}\cos(2x)+\frac{k_a+k_b}{2}\right)\left(-\frac{\sin(\cos(x))}{\sin(x)}+\sin(1)\cot(x)\right)+J_0(1)x.
\end{equation}
In this case, the derivation of a bound for the error is more challenging than that of Section \ref{sec:erran}. However, as expected, numerical results reported in Table \ref{tb:result} show that $C_2(x)$ is more accurate than $C_1(x)$.

\section{Numerical experiments}
\label{sec:num}
In this section the accuracy of the proposed approximations is studied. 
In our tests, a set of six different lower/upper integration bounds is considered, and for each $[l,u]$ 50 random intervals $I\subseteq[l,u]$ are taken. The definite integral for each interval is computed with the two proposed approximations $C_1(x)$ \eqref{finalcos} and $C_2(x)$ \eqref{finalcos2}. The average computational times and errors in the six cases are reported. 
The adaptive Cavalieri-Simpson rule with tolerance of $10^{-10}$ is employed as a reference to compute the error. Concerning the computational time, the adaptive Cavalieri-Simpson rule with tolerance of $10^{-6}$ is used. This adaptive rule is similar to the generalized one, where one divides the interval of integration in $n$ equal subintervals, and apply the Cavalieri-Simpson rule for each one, but in the adaptive case the size of the subintervals is different. Roughly speaking, there will be smaller subintervals where the function to be integrated is more ``irregular'' \cite{numanal,simpsad, numrec}. 
The computational results, presented in Table \ref{tb:result}, have been obtained on a 2.8 GHz Intel Core i7 CPU of a computer with 8 GB RAM running Windows 7 and Matlab\textsuperscript{\textregistered}7.01 R14.

\begin{table}[h]

	\caption{Average computational times and errors obtained with the proposed approximations. For each pair of lower/upper integration bounds, 50 random intervals are taken. Reference time is given by the Cavalieri-Simpson adaptive method with tolerance $10^{-6}$, while errors are computed using as a reference the Cavalieri-Simpson adaptive method with tolerance $10^{-10}$.} 
	
	\begin{center} \footnotesize
		\begin{tabular}{|l|rr|rr|r|}
			\hline
			& \multicolumn{2}{c|}{$C_1(x)$} &
			\multicolumn{2}{c|}{$C_2(x)$}& Ref. \\ 
			Bounds& Time $[s]$&Error&Time $[s]$&Error&Time $[s]$\\ \hline
			$[-1,1]$& $0$& $1.7396\cdot10^{-4}$ &$9.4\cdot10^{-4}$ &$3.3839\cdot10^{-6}$&$8.06\cdot10^{-3}$\\ 
			$[-5,5]$& $0$& $2.2042\cdot10^{-4}$ &$0$ &$2.6032\cdot10^{-6}$&$6.272\cdot10^{-2}$\\ 
			$[-10,10]$& $0$& $1.8882\cdot10^{-4}$ &$1.6\cdot10^{-2}$ &$2.9653\cdot10^{-6}$&$1.5386\cdot10^{-1}$\\ 
			$[-20,20]$& $0$& $1.9543\cdot10^{-4}$ &$0$ &$3.5478\cdot10^{-6}$&$6.2846\cdot10^{-1}$\\ 
			$[-50,50]$& $3.2\cdot10^{-4}$& $2.1629\cdot10^{-4}$ &$3.01\cdot10^{-2}$ &$2.8444\cdot10^{-6}$&$1.5231$\\ 
			$[-100,100]$& $2.2\cdot10^{-4}$& $2.3178\cdot10^{-4}$ &$2\cdot10^{-3}$ &$2.6907\cdot10^{-6}$&$2.7107$\\ \hline
		\end{tabular}
	\end{center}
	\label{tb:result}
\end{table}

\section{Conclusions and future work}
\label{sec:concl}
In this paper, two approximations for the integral of $\cos(\cos(x))$ are proposed. An upper bound for the error of the first approximation is derived. Even though such bound cannot be easily computed in the second case, experimental results show that the second approximation outperforms the first one in terms of error, which is two order of magnitude smaller. 

The proposed approximations present three advantages. First, they are expressed by means of mathematical functions, so they could be used for intermediate simplifications and symbolic computation of more complicated integrals. This could be useful for problems arising in the domains like Physics and in situations where an explicit formula is required. Secondly, the errors do not depend on the width of the interval of integration. Finally, the computation time is almost constant, while the time increases drastically for numerical methods like the adaptive Cavalieri-Simpson, as can be seen in Table \ref{tb:result}.

%Even though state-of-the-art mathematical softwares as Maple\textsuperscript{\texttrademark} \cite{maple} or Mathematica\textsuperscript{\textregistered} \cite{wolfram} can solve numerically such integrals very efficiently, it would be interesting to take into account these approximations in order to perform symbolic computation of difficult integrals involving these composite sinusoidal functions. This could be useful for problems arising in the domains like Physics.

As future work, a similar methodology can be used to derive approximations for other integrals of sinusoidal composite functions. Some of them are straightforward to derive, e.g., the integral of $A\cos(\cos(mx+q))$, with $A,m,q\in \mathbb{R}$ (it is sufficient to multiply the proposed approximations by $\frac{A}{m}$). %change of variables
Others may be more challenging, for example the integrals of $\sin(\sin(x))$, $\cos(\sin(x))$ and similar ones. At the same time, the methodology presented in this paper can provide a first roadmap towards their approximation.

\section*{Acknowledgments}
The research published here was partially conducted at the Future Resilient Systems at the Singapore-ETH Centre (SEC). The SEC was established as a collaboration between ETH Zurich and National Research Foundation (NRF) Singapore (FI 370074011) under the auspices of the NRF's Campus for Research Excellence and Technological Enterprise (CREATE) programme.	Financial support by Grants Digiteo 2009-14D ``RMNCCO'' and Digiteo 2009-55D ``ARM'' is gratefully acknowledged.
%The author thanks Simone Buso, Sonia Cafieri, Silvia Lenzi, Leo Liberti, Pierre Hansen and Stefano Pinzoni for their precious comments and suggestions.

% \linenumbers




%% References


%\bibliographystyle{elsarticle-num}






% \end{thebibliography}

\end{document}